\documentclass[12pt,a4paper]{article}

\usepackage{amsmath,amssymb,amsthm,amsfonts}
\usepackage{multirow}

\theoremstyle{definition}
\newtheorem{definition}{Definition}[section]
\newtheorem{proposition}[definition]{Proposition}
\newtheorem{theorem}[definition]{Theorem}
\newtheorem{lemma}[definition]{Lemma}
\newtheorem{corollary}[definition]{Corollary}

\theoremstyle{remark}
\newtheorem{remark}[definition]{Remark}

\newcommand{\Z}{\mathbb{Z}}
\newcommand{\even}{\bar{0}}
\newcommand{\odd}{\bar{1}}
\newcommand{\LL}{\mathrm{L}}
\newcommand{\TT}{\mathrm{T}}
\newcommand{\GG}{\mathrm{G}}

\begin{document}
\title{The automorphism group functor of the $N=4$ Lie conformal superalgebra}
\author{Zhihua Chang\footnote{zhihuachang@gmail.com}\\{\small Department of Mathematics, Bar-Ilan University, Ramat-Gan, 52900, Israel.}}
\maketitle

\begin{abstract}
In this paper, we study the structure and representability of the automorphism group functor of the $N=4$ Lie conformal superalgebra over an algebraically closed field $k$ of characteristic zero.
\end{abstract}

\medskip


\section{Introduction}
Let $k$ be an algebraically closed field of characteristic zero and $A$ a finite-dimensional (associative, Lie, Jordan...) algebra over $k$. The automorphism group $\mathbf{Aut}(A)$ of $A$ is an affine group scheme of finite type over $k$, which is a representable functor from the category $k\text{\bf-rng}$ of commutative associative $k$--algebras to the category {\bf grp} of groups, assigning a commutative associative $k$--algebra $R$ the group of $R$--linear automorphisms of $A\otimes_kR$.

The geometric properties of $\mathbf{Aut}(A)$ are strongly related to the algebraic properties of $A$. For instance,
if $\mathbf{Aut}(A)$ is smooth and connected, then all twisted loop algebras based on $A$ are trivial, i.e., isomorphic to $A\otimes_kk[t^{\pm1}]$ as algebras over the ring $k[t^{\pm1}]$ (see \cite{P2005}).

In the case of Lie algebras, the centreless core of an affine Kac-Moody Lie algebra is a twisted loop algebra based on a finite-dimensional simple Lie algebra $\frak{g}$ over $k$ (see \cite[Chapters 7 and 8]{K1}). Based on the point of view taken in \cite{P2005}, such a twisted loop algebra is a twisted form of the untwisted loop algebra $\frak{g}\otimes_kk[t^{\pm1}]$ with respect to an \'{e}tale extension of $k[t^{\pm1}]$. These twisted forms are classified in \cite{P2005} by investigating the non-abelian \'{e}tale cohomology of the $k$--group scheme $\mathbf{Aut}(\frak{g})$. The case of  finite-dimensional simple Lie superalgebras was studied in \cite{GP1} and \cite{GP2} as well.

In \cite{KLP}, in order to address the results about the $N=2,4$ superconformal algebras of \cite{SS}, an analogous theory of twisted forms of differential conformal superalgebras is developed. To obtain the ``correct'' definition of conformal superalgebra over a ring, $k$--differential rings were considered as a replacement of commutative associative $k$--algebras. A $k$--differential ring is a pair $\mathcal{R}=(R,\delta)$ consisting of a commutative associative $k$--algebra $R$ and a $k$--linear derivation $\delta:R\rightarrow R$. $k$--differential rings form a category, denoted by $k\text{\bf-drng}$, where a morphism $f:\mathcal{R}=(R,\delta_R)\rightarrow\mathcal{S}=(S,\delta_S)$ is a homomorphism of $k$-algebras $f:R\rightarrow S$ such that $f\circ\delta_R=\delta_S\circ f$. For a conformal superalgebra $\mathcal{A}$ over $k$, its automorphism group functor $\mathbf{Aut}(\mathcal{A})$ is a functor from $k\text{\bf-drng}$ to the category of groups (denoted by {\bf grp}). We will review the precise definition of a Lie conformal superalgebra over $k$ and its automorphism group functor in Section~\ref{sec:basic}.

Unlike the case of usual finite-dimensional algebras, for an arbitrary Lie conformal superalgebra $\mathcal{A}$, $\mathbf{Aut}(\mathcal{A})$ fails to be representable. However, it is still significant for studying the Lie conformal superalgebra $\mathcal{A}$: it is pointed out in \cite{KLP} that the isomorphism classes of twisted forms of $\mathcal{A}\otimes_k\mathcal{R}$ with respect to a faithfully flat extension of $k$--differential rings $\mathcal{R}\rightarrow\mathcal{S}$ are parameterized by the non-abelian cohomology set $\mathrm{H}^1\left(\mathcal{S}/\mathcal{R},\mathbf{Aut}(\mathcal{A})_{\mathcal{R}}\right)$. In particular, a consideration on the extension\footnote{The $k$--differential ring $\widehat{\mathcal{D}}$ plays the role of a simply-connected cover of $\mathcal{D}$.} of $k$--differential rings:
$$\textstyle{\mathcal{D}:=(k[t^{\pm1}],\frac{d}{dt})\rightarrow\widehat{\mathcal{D}}:=(\lim\limits_{\longrightarrow_m}k[t^{\pm\frac{1}{m}}],\frac{d}{dt})}$$
will lead to the classification of twisted loop conformal superalgebras based on $\mathcal{A}$.

In order to completely classify twisted loop conformal superalgebras based on one of the $N=2$ and $N=4$ conformal superalgebras $\mathcal{A}$, the automorphism group $\mathbf{Aut}(\mathcal{A})(\widehat{\mathcal{D}})$ has been determined in \cite{KLP}. The automorphism group functors of the $N=1,2,3$ conformal superalgebras were studied in \cite{CP} as well. Interestingly, for the $N=1,2,3$ Lie conformal superalgebra $\mathcal{K}_N$, $\mathbf{Aut}(\mathcal{K}_N)$ has a subgroup functor $\mathbf{GrAut}(\mathcal{K}_N)$ such that $\mathbf{GrAut}(\mathcal{K}_N)(\mathcal{R})=\mathbf{Aut}(\mathcal{K}_N)(\mathcal{R})$ for $\mathcal{R}=(R,\delta)$ with $R$ an integral domain, and $\mathbf{GrAut}(\mathcal{K}_N)$ is the group functor obtained by lifting the $k$--group scheme $\mathbf{O}_N$ of $N\times N$--orthogonal matrices via the forgetful functor $\mathrm{fgt}:k\text{\bf-drng}\rightarrow k\text{\bf-rng}, \mathcal{R}=(R,\delta)\mapsto R$, i.e.,
$$\mathbf{GrAut}(\mathcal{K}_N):k\text{\bf-drng}\overset{\mathrm{fgt}}{\longrightarrow}
k\text{\bf-rng}\overset{\mathbf{O}_N}{\longrightarrow}\text{\bf grp}.$$
We call a functor from $k\text{\bf-drng}$ to $\text{\bf grp}$ {\it forgetfully representable} if it is obtained by lifting a $k$-group scheme using the forgetful functor.

It has been conjecture in \cite{CP} that the automorphism group functor $\mathbf{Aut}(\mathcal{F})$ of the $N=4$ Lie conformal superalgebra $\mathcal{F}$ also satisfies similar properties: \begin{enumerate}
\item[(1)] It has a subgroup functor $\mathbf{GrAut}(\mathcal{F})$ such that $\mathbf{GrAut}(\mathcal{F})(\mathcal{R})=\mathbf{Aut}(\mathcal{F})(\mathcal{R})$ if $\mathcal{R}=(R,\delta)$ such that $R$ is an integral domain.
\item[(2)] $\mathbf{GrAut}(\mathcal{F})$ is forgetfully representable by
$$\frac{\mathbf{SL}_2\times\mathbf{SL}_2(k)}{\langle-I_2,-I_2\rangle},$$
where $\mathbf{SL}_2$ is the group scheme of $2\times2$ matrices with determinant $1$, and $\mathbf{SL}_2(k)$ is the constant group scheme.
\end{enumerate}

In this paper, we will show that the assertion (1) is true in Section~\ref{sec:auto}, but the assertion (2) fails as we shall see in Remark~\ref{rmk:cex_fgt_rep}. Nonetheless, the subgroup functor $\mathbf{GrAut}(\mathcal{F})$ is still presented by certain affine group schemes. The key ingredient of this paper is to study the relationship between $\mathbf{GrAut}(\mathcal{F})$ and certain affine group schemes (c.f. Section~\ref{sec:grauto}). Finally, we will provide an example to show that the definition of the subgroup functor $\mathbf{GrAut}(\mathcal{A})$ of a Lie conformal superalgebra $\mathcal{A}$ depends on the choice of generators of $\mathcal{A}$ in Section~\ref{sec:grautodef}.

\bigskip

\noindent{\bf Notation: }Throughout this paper, $\Z,\Z_+$, and $\Bbb Q$ denote the sets of integers, non-negative integers, and rational numbers, respectively.

We will always assume $k$ is an algebraically closed field of characteristic zero, $\mathbf{i}$ is $\sqrt{-1}\in k$, and $k[\partial]$ is the polynomial ring over $k$ in one variable $\partial$. We will use $\epsilon_{ijl}$ to denote the sign of a cycle $(ijl)$. For a commutative associative $k$-algebra $R$, $\mathrm{Mat}_2(R)$ will denote the associative algebra of all $2\times2$--matrices with entries in $R$.

We will also use the standard notation of the group scheme $\boldsymbol{\mu_2}$, which is a functor from $k\text{\bf-rng}$ to $\text{\bf grp}$ assigning each object $R$ in $k\text{\bf-rng}$ the group of square roots of unity in $R$, i.e., $\boldsymbol{\mu_2}(R)=\{r\in R|r^2=1\}$.

\section{Preliminaries}
\label{sec:basic}
In this section, we will review some basic concepts of Lie conformal superalgebras over $k$, their base change with respect to an extension of $k$--differential rings, and their automorphism group functors. We then describe the $N=4$ Lie conformal superalgebra using generators and relations. We also introduce some notation to simplify our computations for automorphism groups.

\begin{definition}[{\cite[definition~2.7]{K}}]
A {\it Lie conformal superalgebra} over $k$ is a $\Z/2\Z$--graded $k[\partial]$--module $\mathcal{A}=\mathcal{A}_{\even}\oplus\mathcal{A}_{\odd}$ equipped with a $k$--bilinear operation $-_{(n)}-:\mathcal{A}\otimes_{k}\mathcal{A}\rightarrow\mathcal{A}$ for each $n\in\Z_+$ satisfying the following axioms:
\begin{quote}
\begin{itemize}
\item[$\mathrm{(CS0)}$] $a_{(n)}b=0$ for $n\gg0$,
\item[$\mathrm{(CS1)}$] $\partial(a)_{(n)}b=-na_{(n-1)}b$,
\item[$\mathrm{(CS2)}$] $a_{(n)}b=-p(a,b)\sum_{j\in\Z_+}(-1)^{j+n}\partial^{(j)}(b_{(n+j)}a)$,
\item[$\mathrm{(CS3)}$] $a_{(m)}(b_{(n)}c)=\sum_{j=0}^m\binom{m}{j}(a_{(j)}b)_{(m+n-j)}c+p(a,b)b_{(n)}(a_{(m)}c)$,
\end{itemize}
\end{quote}
where $\partial^{(j)}=\partial^{j}/j!, j\in\Bbb Z_+$ and $p(a,b)=(-1)^{p(a)p(b)}$, $p(a)$ (resp. $p(b)$) is the parity of $a$ (resp. $b$).
\end{definition}

We also use the conventional $\lambda$-bracket notation:
$$[a_\lambda b]=\sum\limits_{n\in\Z_+}\frac{1}{n!}\lambda^na_{(n)}b,$$
where $\lambda$ is a variable and $a,b\in\mathcal{A}$.

Given a Lie conformal superalgebra $\mathcal{A}$ over $k$ and a $k$--differential ring $\mathcal{R}=(R,\delta)$, we can form a new Lie conformal superalgebra $\mathcal{A}\otimes_k\mathcal{R}$ (the base change of $\mathcal{A}$ with respect to $k\rightarrow\mathcal{R}$) as follows: The underlying $\Z/2\Z$--graded $k$--vector space of $\mathcal{A}\otimes_k\mathcal{R}$ is $\mathcal{A}\otimes_kR$, the $k[\partial]$--module structure is given by $\partial_{\mathcal{A}\otimes_k\mathcal{R}}=\partial_{\mathcal{A}}\otimes1+1\otimes\delta$, and the $n$-th product for each $n\in\Z_+$ is given by:
$$(a\otimes r)_{(n)}(b\otimes s)=\sum\limits_{j\in\Z_+}(a_{(n+j)}b)\otimes\delta^{(j)}(r)s,$$
where $a,b\in\mathcal{A}, r,s\in R$ and $\delta^{(j)}=\delta^j/j!$.

It is pointed out in \cite{KLP} that $\mathcal{A}\otimes_k\mathcal{R}$ is not only a Lie conformal superalgebra over $k$, but also an $\mathcal{R}$--Lie conformal superalgebra (c. f. \cite[definition~1.3]{KLP}). The $\mathcal{R}$--structure yields the definition of its $\mathcal{R}$--automorphisms: an {\it$\mathcal{R}$--automorphism} of $\mathcal{A}\otimes_k\mathcal{R}$ is an automorphism\footnote{Throughout this paper, an automorphism of a $\Z/2\Z$--graded $R$--module is required to be an automorphism of degree $\bar{0}$, i.e., it preserves the $\Z/2\Z$--gradation.}  of the $\Z/2\Z$--graded $R$--module $\mathcal{A}\otimes_kR$ which preserves all $n$-th products and commutes with $\partial_{\mathcal{A}\otimes_k\mathcal{R}}$. All $\mathcal{R}$--automorphisms of $\mathcal{A}\otimes_k\mathcal{R}$ form a group, denoted by $\mathrm{Aut}_{\mathcal{R}\text{-conf}}(\mathcal{A}\otimes_k\mathcal{R})$. These automorphism groups are functorial in $\mathcal{R}$. Hence, one defines a functor from the category of $k$--differential rings to the category of groups:
$$\mathbf{Aut}(\mathcal{A}):k\text{\bf-drng}\rightarrow\text{\bf grp},\quad \mathcal{R}\mapsto\mathrm{Aut}_{\mathcal{R}\text{-conf}}(\mathcal{A}\otimes_{k}\mathcal{R}).$$
It is called the {\it automorphism group functor} of $\mathcal{A}$.
\bigskip

In this paper, we focus on the $N=4$ Lie conformal superalgebra $\mathcal{F}$. As a $\Z/2\Z$--graded $k[\partial]$--module, $\mathcal{F}=\mathcal{F}_{\even}\oplus\mathcal{F}_{\odd}$,
where
\begin{align*}
\mathcal{F}_{\even}&=k[\partial] \LL\oplus k[\partial] \TT^1\oplus k[\partial] \TT^2\oplus k[\partial]\TT^3, \\ \mathcal{F}_{\odd}&=k[\partial] \GG^1\oplus k[\partial] \GG^2\oplus k[\partial] \overline{\GG}^1\oplus k[\partial]\overline{\GG}^2.
\end{align*}
The $\lambda$--bracket on $\mathcal{F}$ is given by
\begin{center}
\begin{tabular}{ll}
$[\LL_\lambda \LL]=(\partial+2\lambda)\LL$,
&$[\LL_\lambda \TT^i]=(\partial+\lambda)\TT^i$,\\[1ex]
$[\LL_\lambda \GG^p]=\left(\partial+\frac{3}{2}\lambda\right)\GG^p$,
&$\left[\LL_\lambda \overline{\GG}^p\right]=\left(\partial+\frac{3}{2}\lambda\right)\overline{\GG}^p$,\\[1ex]
$[{\TT^i}_\lambda \TT^j]=\mathbf{i} \epsilon_{ijl}\TT^l$,
&$[{\GG^p}_\lambda \GG^q]=\left[\hspace*{0.5mm}{\overline{\GG}^p}_\lambda\overline{\GG}^q\right]=0$,\\
$[{\TT^i}_\lambda \GG^p]=-\frac{1}{2}\sum\limits_{q=1}^2\sigma_{pq}^i\GG^q$,
&$\left[{\TT^i}_\lambda \overline{\GG}^p\right]=\frac{1}{2}\sum\limits_{q=1}^2\sigma_{qp}^i\overline{\GG}^q$,\\[0.5ex]
\multicolumn{2}{l}{$\left[{\GG^p}_\lambda \overline{\GG}^q\right]=2\delta_{pq}\LL-2(\partial+2\lambda)\sum\limits_{i=1}^3\sigma_{pq}^i\TT^i$,}
\end{tabular}
\end{center}
where $i,j=1,2,3$, $p,q=1,2$ and $\sigma^i, i=1,2,3$ are the Pauli spin matrices:
$$\sigma^1=\begin{pmatrix}0&1\\1&0\end{pmatrix},\quad \sigma^2=\begin{pmatrix}0&-\mathbf{i}\\ \mathbf{i}&0\end{pmatrix},
\quad \sigma^3=\begin{pmatrix}1&0\\0&-1\end{pmatrix}.$$

The even part $\mathcal{F}_{\even}$ has a sub Lie conformal algebra $$\mathcal{B}=k[\partial]\TT^1\oplus k[\partial]\TT^2\oplus k[\partial]\TT^3,$$ which is isomorphic to the current Lie conformal algebra\footnote{Every finite-dimensional Lie superalgebra $\frak{g}$ over $k$ associates a current Lie conformal superalgebra $\mathrm{Cur}(\frak{g})=k[\partial]\otimes_k\frak{g}$ with the $n$-th product: $a_{(0)}b=[a,b]$ and $a_{(n)}b=0$ for $n\geqslant1$ \cite[Example~2.7]{K}.} $\mathrm{Cur}(\frak{sl}_2(k))$.

Analogous to the cases where $N=1,2,3$, the automorphism group functor $\mathbf{Aut}(\mathcal{F})$ has a subgroup functor $\mathbf{GrAut}(\mathcal{F})$ defined as follows: fix the $k$--vector space
\begin{equation}
V=\mathrm{span}_k\left\{\LL,\TT^1,\TT^2,\TT^3,\GG^1,\GG^2,\overline{\GG}^1,\overline{\GG}^2\right\},\label{eq:vectorspace}
\end{equation}
and define a subgroup of $\mathbf{Aut}(\mathcal{A})(\mathcal{R})$ for each $k$--differential ring $\mathcal{R}=(R,\delta)$:
$$\mathbf{GrAut}(\mathcal{F})(\mathcal{R})=\left\{\phi\in\mathbf{Aut}(\mathcal{F})(\mathcal{R})|\phi(V\otimes_k R)\subseteq V\otimes_kR\right\}.$$
This construction is functorial in $\mathcal{R}$.

To simplify computations for these automorphism groups, we write $\widehat{\partial}=\partial_{\mathcal{F}\otimes_k\mathcal{R}}$ for short and introduce the following notation\footnote{These notations allow us to rewrite all commutative relations of $\mathcal{F}$ using matrix multiplication, which lead to nice formulas for describing the automorphisms of $\mathcal{F}$}: for $r\in R$, $X=\begin{pmatrix}x&y\\ z&-x\end{pmatrix}\in\frak{sl}_2(R)$ and $M=\begin{pmatrix}a&b\\c&d\end{pmatrix}\in \mathrm{Mat}_2(R)$, we set
\begin{align*}
&\LL(r):=\LL\otimes r,\\
&\TT(X):=\TT^1\otimes(y+z)+\mathbf{i} \TT^2\otimes(y-z)+2 \TT^3\otimes x,\\
&\GG(M):=\GG^1\otimes d+\overline{\GG}^1\otimes a-\GG^2\otimes b+\overline{\GG}^2\otimes c.
\end{align*}
Then every element of $V$ (resp. $V\otimes R$) can be written as
$$\LL(r)+\TT(X)+\GG(M),$$
where $r\in k, X\in\frak{sl}_2(k), M\in \mathrm{Mat}_2(k)$ (resp. $r\in R, X\in\frak{sl}_2(R), M\in \mathrm{Mat}_2(R)$). Furthermore,  the $\lambda$-bracket on $\mathcal{F}\otimes_{k}{\mathcal{R}}$ satisfies the following relations: for $r,r'\in R$, $X,X_1,X_2\in\frak{sl}_2(R)$, and $M,N\in \mathrm{Mat}_2(R),$
\begin{align*}
\/[\LL(r)_\lambda \LL(r')]&=(\partial+2\lambda)\LL(rr')+2\LL(\delta(r)r'),\\
\/[\LL(r)_\lambda \TT(X)]&=(\partial+\lambda)\TT(rX)+\TT(\delta(r)X),\\
\/[\TT(X_1)_\lambda \TT(X_2)]&=\TT([X_1,X_2]),\\
\/[\LL(r)_\lambda \GG(M)]&=\left(\partial+\frac{3}{2}\lambda\right)\GG(rM)+\frac{3}{2}\GG(\delta(r)M),\\
\/[\TT(X)_\lambda \GG(M)]&=\GG(XM),\\
\/[\GG(M)_\lambda \GG(N)]&=2\LL(\mathrm{tr}(MN^\dag))+(\partial+2\lambda)\TT(MN^\dag-NM^\dag)\\
&\quad+2\TT(\delta(M)N^\dag-N\delta(M^\dag)).
\end{align*}
where $\mathrm{tr}:\mathrm{Mat}_2(R)\rightarrow R$ is the trace map, and
\begin{equation}
\begin{pmatrix}a&b\\c&d\end{pmatrix}^\dag:=\begin{pmatrix}d&-b\\-c&a\end{pmatrix}\label{eq:involution}
\end{equation}
is the standard sympletic involution on $\mathrm{Mat}_2(R)$.

\section{The group functor $\mathbf{GrAut}(\mathcal{F})$}
\label{sec:grauto}
We will compute the group $\mathbf{GrAut}(\mathcal{F})(\mathcal{R})$ for a $k$--differential ring $\mathcal{R}=(R,\delta)$ in this section. The main results will be Theorem~\ref{thm:exact}, Propositions~\ref{pro:surj1} and~\ref{pro:surj2}.

\begin{lemma}
\label{lem:SL2R}
For an arbitrary $k$--differential ring $\mathcal{R}=(R,\delta)$, there is a group homomorphism
$$\iota_{\mathcal{R}}:\mathbf{SL}_2(R)\times\mathbf{SL}_2(R_0)
\rightarrow\mathbf{GrAut}(\mathcal{F})(\mathcal{R}),\quad (A,B)\mapsto \theta_{A,B},$$
where $R_0=\ker\delta$, and $\theta_{A,B}$ is the $\mathcal{R}$--automorphism of $\mathcal{F}\otimes_{k}\mathcal{R}$ defined by
\begin{eqnarray*}
\theta_{A,B}(\LL)=\LL+\TT(\delta(A)A^{-1}),\\
\theta_{A,B}(\TT(X))=\TT(AX A^{-1}),\\
\theta_{A,B}(\GG(M))=\GG(AMB^{-1}),
\end{eqnarray*}
for $X\in\frak{sl}_2(k), M\in \mathrm{Mat}_2(k)$. In addition, the homomorphism $\iota_{\mathcal{R}}$ is functorial in $\mathcal{R}$.
\end{lemma}

\begin{proof}
These formulas define a homomorphism $\theta_{A,B}$ of $R$--modules $V\otimes_{k}R\rightarrow V\otimes_{k}R$. It is extended to a homomorphism of $R$--modules $\mathcal{F}\otimes_{k}\mathcal{R}\rightarrow\mathcal{F}\otimes_{k}\mathcal{R}$ which preserves the $\Z/2\Z$--grading and satisfies $\widehat{\partial}\circ\theta_{A,B}=\theta_{A,B}\circ\widehat{\partial}$.
This map is also denoted by $\theta_{A,B}$.

To show $\theta_{A,B}$ is a homomorphism of Lie conformal superalgebras, by Lemma~3.1 (ii) in \cite{KLP}, it suffices to show
\begin{align}
\theta_{A,B}([\xi\otimes1_\lambda\eta\otimes1])=[\theta_{A,B}(\xi\otimes1)_\lambda\theta_{A,B}(\eta\otimes1)]\label{eq:preprod}
\end{align}
for all $\xi,\eta\in V$. This can be accomplished by a direct computation.

For instance, let $M,N\in \mathrm{Mat}_2(k)$, then
\begin{align*}
&\theta_{A,B}([\GG(M)_\lambda\GG(N)])\\
&=\theta_{A,B}\left(2\LL(\mathrm{tr}(MN^\dag))+(\partial+2\lambda)\TT(MN^\dag-NM^\dag)\right)\\
&=2\LL(\mathrm{tr}(MN^\dag))+2\TT(\mathrm{tr}(MN^\dag)\delta(A)A^{-1})\\
&\quad+(\widehat{\partial}+2\lambda)\TT(A(MN^\dag-NM^\dag)A^{-1}),\\
&[\theta_{A,B}(\GG(M))_\lambda\theta_{A,B}(\GG(N))]\\
&=[\GG(AMB^{-1})_\lambda\GG(ANB^{-1})]\\
&=2\LL(\mathrm{tr}(AMB^{-1}(ANB^{-1})^\dag))\\
&\quad+(\partial+2\lambda)\TT(AMB^{-1}(ANB^{-1})^\dag-ANB^{-1}(AMB^{-1})^\dag)\\
&\quad+2\TT(\delta(AMB^{-1})(ANB^{-1})^\dag-ANB^{-1}\delta(AMB^{-1})^\dag).
\end{align*}
For the standard sympletic involution as defined in (\ref{eq:involution}), we have
\begin{itemize}
\item $(MN)^{\dag}=N^{\dag}M^{\dag}$ for $M,N\in \mathrm{Mat}_2(R)$.
\item $MN^\dag+NM^\dag=\mathrm{tr}(MN^\dag)I$ for $M,N\in \mathrm{Mat}_2(R)$, where $I$ is the identity matrix.
\item $A^{-1}=A^\dag$ for $A\in \mathbf{SL}_2(R)$.
\end{itemize}
Hence, $AMB^{-1}(ANB^{-1})^\dag=AMB^{-1}BN^\dag A^{-1}=AMN^\dag A^{-1}$. Noting that $\delta(B)=0$ and $\delta(A^{-1})=-A^{-1}\delta(A)A^{-1}$, we obtain
\begin{align*}
&2(\delta(AMB^{-1})(ANB^{-1})^\dag-ANB^{-1}\delta(AMB^{-1})^\dag)\\
&=2(\delta(A)MB^{-1}BN^\dag A^{-1}-ANB^{-1}BM^\dag\delta(A^{-1}))\\
&=\delta(A)(MN^\dag-NM^\dag) A^{-1}+A(MN^\dag-NM^\dag)\delta(A^{-1})\\
&\quad+\delta(A)(MN^\dag+NM^\dag) A^{-1}-A(MN^\dag+NM^\dag)\delta(A^{-1})\\
&=\delta(A(MN^\dag-NM^\dag)A^{-1})+\mathrm{tr}(MN^\dag)(\delta(A)A^{-1}-A\delta(A^{-1}))\\
&=\delta(A(MN^\dag-NM^\dag)A^{-1})+2\mathrm{tr}(MN^\dag)\delta(A)A^{-1}.
\end{align*}
It follows that
$$\theta_{A,B}([\GG(M)_\lambda\GG(N)])=[\theta_{A,B}(\GG(M))_\lambda\theta_{A,B}(\GG(N))].$$

Similarly, it is easy to verify that the equation (\ref{eq:preprod}) holds for all $\xi,\eta\in V$. Hence, $\theta_A$ is a homomorphism of $\mathcal{R}$--Lie conformal superalgebras. Then $\theta_{A,B}\circ\theta_{A^{-1},B^{-1}}=\theta_{A^{-1},B^{-1}}\circ\theta_{A,B}=\mathrm{id}_{\mathcal{F}}$ implies that $\theta_{A,B}$ is an automorphism of the $\mathcal{R}$--Lie conformal superalgebra $\mathcal{F}\otimes_k\mathcal{R}$. Hence, we obtain $\theta_{A,B}\in\mathbf{GrAut}(\mathcal{F})(\mathcal{R})$ since $\theta_{A,B}(V\otimes_{k}R)\subseteq V\otimes_{k}R$.

To prove $(A,B)\mapsto\theta_{A,B}$ defines a group homomorphism, it suffices to show $\theta_{A_1A_2,B_1B_2}=\theta_{A_1,B_1}\circ\theta_{A_2,B_2}$ for $A_1,A_2\in\mathbf{SL}_2(R), B_1,B_2\in\mathbf{SL}_2(R_0)$. By \cite[Proposition~3.1]{KLP}, this can be easily done by verifying $$\theta_{A_1A_2,B_1B_2}(\xi\otimes1)=\theta_{A_1,B_1}\circ\theta_{A_2,B_2}(\xi\otimes1)$$
for $\xi\in V$. For instance,
\begin{align*}
\theta_{A_1A_2,B_1B_2}(\LL)&=\LL+\TT(\delta(A_1A_2)(A_1A_2)^{-1})\\
&=\LL+\TT((\delta(A_1)A_2+A_1\delta(A_2))A_2^{-1}A_1^{-1})\\
&=\LL+\TT(\delta(A_1)A_1^{-1}+A_1\delta(A_2)A_2^{-1}A_1^{-1}),\\
\theta_{A_1,B_1}\circ\theta_{A_2,B_2}(\LL)&=\theta_{A_1,B_1}(\LL+\TT(\delta(A_2)A_2^{-1}))\\
&=\LL+\TT(\delta(A_1)A_1^{-1})+\TT(A_1\delta(A_2)A_2^{-1}A_1^{-1}).
\end{align*}

In summary, we obtain a group homomorphism $$\iota_{\mathcal{R}}:\mathbf{SL}_2(R)\times\mathbf{SL}_2(R_0)\rightarrow\mathbf{GrAut}(\mathcal{F})(\mathcal{R}),$$
for each $k$--differential ring $\mathcal{R}$ which is functorial in $\mathcal{R}$.
\end{proof}

\begin{lemma}
\label{lem:center}
For every $k$--differential ring $\mathcal{R}=(R,\delta)$,
$$\ker(\iota_{\mathcal{R}})\cong\boldsymbol{\mu_2}(R_0),$$
where $R_0=\ker\delta$ and $\boldsymbol{\mu_2}(R_0)$ is diagonally embedded into $\mathbf{SL}_2(R)\times\mathbf{SL}_2(R_0)$. Additionally, these isomorphisms are functorial in $\mathcal{R}$.
\end{lemma}
\begin{proof}
Let $(A,B)\in\ker(\iota_{\mathcal{R}})$, where $A\in\mathbf{SL}_2(R)$ and $B\in\mathbf{SL}_2(R_0)$. Since $\theta_{A,B}=\mathrm{id}$,
$$\TT(X)=\theta_{A,B}(\TT(X))=\TT(AXA^{-1}),$$
for all $X\in\frak{sl}_2(k)$. Hence, $AX=XA$ for all $X\in\frak{sl}_2(k)$. It follows that $A=aI_2$ for some $a\in\boldsymbol{\mu}_2(R)$.

Since
$$\GG(M)=\theta_{A,B}(\GG(M))=\GG(AMB^{-1}),$$
for all $M\in \mathrm{Mat}_2(k)$, we obtain $aM=AM=MB$ for all $M\in \mathrm{Mat}_2(k)$, and hence $B=aI_2$ and $a\in R_0$ as $B\in\mathbf{SL}_2(R_0)$. Therefore, $(A,B)=(aI_2,aI_2)$ for $a\in\boldsymbol{\mu}_2(R_0)$.

Conversely, for $a\in\boldsymbol{\mu}_2(R_0)$, it can be checked that $(aI_2,aI_2)\in\ker(\iota_{\mathcal{R}})$. Hence, $\ker(\iota_{\mathcal{R}})\cong\boldsymbol{\mu}_2(R_0)$.
\end{proof}

Summarizing Lemmas \ref{lem:SL2R} and \ref{lem:center}, we obtain the following theorem:
\begin{theorem}
\label{thm:exact}
For every $k$--differential ring $\mathcal{R}=(R,\delta)$, there is an exact sequence of groups
\begin{equation}
1\rightarrow\boldsymbol{\mu}_2(R_0)\rightarrow
\mathbf{SL}_2(R)\times\mathbf{SL}_2(R_0)\xrightarrow{\iota_{\mathcal{R}}}
\mathbf{GrAut}(\mathcal{F})(\mathcal{R}),\label{eq:exact}
\end{equation}
where $R_0=\ker\delta$. Furthermore, the exact sequence is functorial in $\mathcal{R}$.\qed
\end{theorem}

In general, $\iota_{\mathcal{R}}$ fails to be surjective. However, it has properties analogous to the universal surjectivity of the quotient morphisms of $k$--group schemes. More precisely, we have the following Propositions~\ref{pro:surj1} and~\ref{pro:surj2}. We call an extension $\mathcal{R}=(R,\delta_R)\rightarrow\mathcal{S}=(S,\delta_S)$ of $k$--differential rings {\it\'{e}tale} if the homomorphism $R\rightarrow S$ of rings is \'{e}tale.

\begin{proposition}
\label{pro:surj1}
Let $\mathcal{R}=(R,\delta_R)$ be a $k$--differential ring with $R$ an integral domain. For every $\varphi\in\mathbf{GrAut}(\mathcal{F})(\mathcal{R})$, there is an \'{e}tale extension $\mathcal{S}=(S,\delta_S)$ of $\mathcal{R}$, an element $A\in\mathbf{SL}_2(S)$, and an element $B\in\mathbf{SL}_2(S_0)$ such that
$$\iota_{\mathcal{S}}(A,B)=\theta_{A,B}=\varphi_{\mathcal{S}},$$
where $S_0=\ker\delta_S$, and $\varphi_{\mathcal{S}}$ is the image of $\varphi$ under the induces group homomorphism $\mathbf{GrAut}(\mathcal{F})(\mathcal{R})\rightarrow\mathbf{GrAut}(\mathcal{F})(\mathcal{S})$.
\end{proposition}
\begin{proof}
We first write $\varphi(\LL)=\LL(r)+\TT(X_0)$,
where $r\in R, X_0\in\frak{sl}_2(R)$. Then
\begin{align*}
\varphi([\LL_\lambda\LL])&=(\widehat{\partial}+2\lambda)(\LL(r)+\TT(X_0)),\\
[\varphi(\LL)_\lambda\varphi(\LL)]&=(\widehat{\partial}+2\lambda)(\LL(r^2)+\TT(rX_0)).
\end{align*}
We deduce from $\varphi([\LL_\lambda\LL])=[\varphi(\LL)_\lambda\varphi(\LL)]$ that $r^2=r$ and $rX_0=X_0$. Since $R$ is an integral domain, $r=0$ or $1$. If $r=0$, we obtain $X_0=rX_0=0$, and so $\varphi(\LL)=0$. This contradicts the injectivity of $\varphi$. Hence, $r=1$, i.e.,
\begin{equation}
\varphi(\LL)=\LL+\TT(X_0).\label{eq:vir}
\end{equation}

Next we write $\varphi(\TT(\sigma^i))=\LL(r_i)+\TT(X_i)$ for $r_i\in R, X_i\in\frak{sl}_2(R),i=1,2,3$. Thus,
\begin{align*}
\varphi([\LL_\lambda\TT(\sigma^i)])&=(\widehat{\partial}+\lambda)(\LL(r_i)+\TT(X_i)),\\
[\varphi(\LL)_\lambda\varphi(\TT(\sigma^i))]&=(\partial+2\lambda)\LL(r_i)+(\partial+\lambda)\TT(X_i)\\
&\quad+\lambda\TT(r_iX_0)+\TT(r_i\delta(X_0))+\TT([X_0,X_i]).
\end{align*}
We deduce from $\varphi([\LL_\lambda\TT(\sigma^i)])=[\varphi(\LL)_\lambda\varphi(\TT(\sigma^i))], i=1,2,3$ that $r_i=0$ and $\delta(X_i)=[X_0,X_i]$, $i=1,2,3$. It follows that $\varphi(\TT(\sigma^i))=\TT(X_i), i=1,2,3$. Hence, $\varphi(\mathcal{B}\otimes_k\mathcal{R})\subseteq\mathcal{B}\otimes_k\mathcal{R}$ and $\varphi|_{\mathcal{B}\otimes_{k}\mathcal{R}}$ is an automorphism of $\mathcal{B}\otimes_{k}\mathcal{R}$.

Recall that $\mathcal{F}_{\even}$ has a subalgebra $\mathcal{B}=k[\partial]\TT^1\oplus k[\partial]\TT^2\oplus k[\partial]\TT^3$, which is isomorphic to $\mathrm{Cur}(\frak{sl}_2(k))$. Thus $\mathcal{B}\otimes_{k}\mathcal{R}\cong\mathrm{Cur}(\frak{sl}_2(k))\otimes_{k}\mathcal{R}$ as $\mathcal{R}$--Lie conformal superalgebras. By Corollary~3.17 of \cite{KLP}, there is an $R$-linear automorphism $\overline{\varphi}$ of the Lie algebra  $\frak{sl}_2(k)\otimes_{k}R=\frak{sl}_2(R)$ such that $\varphi(\TT(X))=\TT(\overline{\varphi}(X))$ for all $X\in\frak{sl}_2(R)$.

It is known that $\mathbf{SL}_2(R)$ acts on $\frak{sl}_2(k)\otimes_kR$ functorially via conjugation and there is a quotient morphism $\pi:\mathbf{SL}_2\rightarrow\mathbf{Aut}(\frak{sl}_2(k))$ of $k$-group schemes. From \cite[III, \S1, 2.8 and 3.2]{DG}, there exist an \'{e}tale extension $S$ of $R$ and $A\in\mathbf{SL}_2(S)$ such that
\begin{equation}
\overline{\varphi}_S(X)=AXA^{-1},\quad X\in\frak{sl}_2(S).\label{eq:sl2conj}
\end{equation}

Since $R\rightarrow S$ is an \'{e}tale homomorphism, by \cite[Chapter 0, Corollary~20.5.8]{grothendieck1964ega4}, there is a unique $k$--derivation $\delta_S$ of $S$ extending $\delta_R$. Hence, $\mathcal{S}=(S,\delta_S)$ is an \'{e}tale extension of $\mathcal{R}$.

Now we consider the image $\varphi_{\mathcal{S}}$ of $\varphi$ in $\mathbf{GrAut}(\mathcal{F})(\mathcal{S})$. From (\ref{eq:vir}) and (\ref{eq:sl2conj}), we obtain
\begin{align*}
\varphi_{\mathcal{S}}(\LL)=\LL+\TT(X_0),\text{ and }\varphi_{\mathcal{S}}(\TT(\sigma^i))=\TT(A\sigma^iA^{-1}), i=1,2,3.
\end{align*}
Then we deduce from $[\varphi_{\mathcal{S}}(\LL)_\lambda\varphi_{\mathcal{S}}(\TT(\sigma^i))]
=(\widehat{\partial}+\lambda)\varphi_{\mathcal{S}}(\TT(\sigma^i))$ that $$\delta_S(A\sigma^iA^{-1})=[X_0,A\sigma^iA^{-1}], i=1,2,3.$$
Hence, a direct computation shows that $X_0=\delta_S(A)A^{-1}$.

Let $\psi=\varphi_{\mathcal{S}}\circ\theta_{A,I}^{-1}$. Then $\psi\in\mathbf{GrAut}(\mathcal{F})(\mathcal{S})$ and $\psi|_{\mathcal{F}_{\even}\otimes_k\mathcal{S}}=\mathrm{id}$. Next we consider $\psi|_{\mathcal{F}_{\odd}\otimes_{k}S}$. Suppose
$$\psi(\GG(M))=\GG(\nu(M)),\quad M\in \mathrm{Mat}_2(S),$$
where $\nu:\mathrm{Mat}_2(S)\rightarrow \mathrm{Mat}_2(S)$ is a bijective $S$--linear map. Now we deduce from
$$\psi([\TT(X)_\lambda\GG(M)])=[\psi(\TT(X))_\lambda\psi(\GG(M))]$$
that $X\cdot\nu(M)=\nu(XM)$ for all $X\in\frak{sl}_2(S)$ and $M\in \mathrm{Mat}_2(S)$. Then a straightforward computation shows that there is an element $B\in\mathbf{GL}_2(S)$ such that $\nu(M)=MB^{-1}$ for all $M\in \mathrm{Mat}_2(S)$.

Next we show $B\in\mathbf{SL}_2(S_0)$. Let $M,N\in \mathrm{Mat}_2(k)$, then
\begin{align*}
\psi([\GG(M)_{(1)}\GG(N)])&=2\TT(MN^\dag-NM^\dag),\\
[\psi(\GG(M))_{(1)}\psi(\GG(N))]&=2\TT(MB^{-1}(NB^{-1})^\dag-NB^{-1}(MB^{-1})^\dag)\\
&=2\TT(\det(B^{-1})(MN^\dag-NM^\dag)).
\end{align*}
It follows that $\det(B^{-1})=1$, so $B\in\mathbf{SL}_2(S)$.

For $M\in \mathrm{Mat}_2(k)$, we consider
\begin{align*}
\psi([\GG(M)_\lambda\LL])&=\left(\frac{1}{2}\widehat{\partial}+\frac{3}{2}\lambda\right)\GG(MB^{-1}),\\
[\psi(\GG(M))_\lambda\psi(\LL)]&=\left(\frac{1}{2}\partial+\frac{3}{2}\lambda\right)\GG(MB^{-1})+\frac{3}{2}\GG(M\delta(B^{-1})).
\end{align*}
These yield that $\delta(B^{-1})=0$, and hence $B\in\mathbf{SL}_2(S_0)$. Therefore, $\psi=\varphi_{\mathcal{S}}\circ\theta_{A,I}^{-1}=\theta_{I,B}$, i.e., $\varphi_{\mathcal{S}}=\theta_{I,B}\circ\theta_{A,I}=\theta_{A,B}$.
\end{proof}

\begin{proposition}
\label{pro:surj2}
Let $\mathcal{R}=(R,\delta)$ be a $k$--differential ring. If $R$ is an integral domain and the \'{e}tale cohomology set $\mathrm{H}_{\text{\'{e}t}}^1(R,\boldsymbol{\mu}_2)$ is trivial, then $\iota_{\mathcal{R}}$ is surjective.
\end{proposition}
\begin{proof}
Given $\varphi\in\mathbf{GrAut}(\mathcal{F})(\mathcal{R})$, as in the proof of Proposition~\ref{pro:surj1}, its restriction to $\mathcal{B}\otimes_k\mathcal{R}$ yields an $R$--linear automorphism $\overline{\varphi}$ of the Lie algebra $\frak{sl}_2(k)\otimes_kR$, where $\mathcal{B}=k[\partial]\TT^1\oplus k[\partial]\TT^2\oplus k[\partial]\TT^3\cong\mathrm{Cur}(\frak{sl}_2(k))$. It is known that there is an short exact sequence of $k$--group schemes
$$1\rightarrow\boldsymbol{\mu_2}\rightarrow\mathbf{SL}_2\rightarrow\mathbf{Aut}(\frak{sl}_2(k))\rightarrow1,$$
which yields a long exact sequences:
$$1\rightarrow \boldsymbol{\mu_2}(R)\rightarrow\mathbf{SL}_2(R)\rightarrow\mathbf{Aut}(\frak{sl}_2(k))(R)\rightarrow
\mathrm{H}^1_{\text{\'{e}t}}(R,\boldsymbol{\mu}_2)\rightarrow\cdots.$$
Hence, the triviality of $\mathrm{H}^1_{\text{\'{e}t}}(R,\boldsymbol{\mu}_2)$ yields that $\overline{\varphi}$ is the image of an element $A\in\mathbf{SL}_2(R)$. i.e.,
$$\varphi(\TT(X))=AXA^{-1}, X\in\frak{sl}_2(R).$$
Then the proof is completed by similar arguments as in Proposition~\ref{pro:surj1}.
\end{proof}

\bigskip

\begin{remark}
\label{rmk:cex_fgt_rep}
We consider the $k$-differential ring $\widehat{\mathcal{D}}'=(k[t^q|q\in\Bbb Q],0)$. Since $k[t^q|q\in\Bbb Q]$ is an integral domain and $\mathrm{H}_{\text{\'{e}t}}^1(\widehat{D},\boldsymbol{\mu_2})$ is trivial (see \cite[Theorem~2.9]{GiP}), we deduce from Proposition~\ref{pro:surj2} that
$$\mathbf{GrAut}(\mathcal{F})(\widehat{\mathcal{D}}')=\frac{\mathbf{SL}_2(\widehat{D})\times\mathbf{SL}_2(\widehat{D})}{\langle-I_2,-I_2\rangle},$$
which shows that the group functor $\mathbf{GrAut}(\mathcal{F})$ fails to be forgetfully representable by $\frac{\mathbf{SL}_2\times\mathbf{SL}_2(k)}{\langle -I_2,-I_2\rangle}$.
\end{remark}

Although the group functor $\mathbf{GrAut}(\mathcal{F})$ is not forgetfully representable, it is strongly related to the group scheme $\mathbf{SL}_2$. In the remaining of this section, we will discuss the representability of group functors involved in the exact sequence (\ref{eq:exact}). Recall that both $\mathbf{SL}_2$ and $\boldsymbol{\mu}_2$ are affine group schemes over $k$ (representable functors from  $k\text{\bf-rng}$ to $\text{\bf grp}$). More precisely,
$$\mathbf{SL}_2=\mathrm{Hom}_{k\text{\bf-rng}}(k[\mathbf{SL}_2],-)\text{ and }
\boldsymbol{\mu}_2=\mathrm{Hom}_{k\text{\bf-rng}}(k[\boldsymbol{\mu}_2],-),$$
where their coordinates rings are $k[\mathbf{SL}_2]=k[x_{11},x_{12},x_{21},x_{22}]/(x_{11}x_{22}-x_{12}x_{21}-1)$
and $k[\boldsymbol{\mu}_2]=k[x]/(x^2-1)$, respectively.

Analogous to the case of $N=1, 2, 3$ Lie conformal algebras \cite{CP}, both $\mathbf{SL}_2$ and $\boldsymbol{\mu}_2$ yield forgetfully representable functors:
\begin{align*}
&\mathbf{SL}_2^{\mathrm{fgt}}:k\text{\bf-drng}\xrightarrow{\mathrm{fgt}}k\text{\bf-rng}
\xrightarrow{\mathbf{SL}_2}\text{\bf grp}, &&\mathcal{R}=(R,\delta)\mapsto\mathbf{SL}_2(R).\\ &\boldsymbol{\mu}_2^{\mathrm{fgt}}:k\text{\bf-drng}\xrightarrow{\mathrm{fgt}}k\text{\bf-rng}
\xrightarrow{\boldsymbol{\mu}_2}\text{\bf grp},&&\mathcal{R}=(R,\delta)\mapsto\boldsymbol{\mu}_2(R).
\end{align*}

However, the group functor $\mathcal{R}=(R,\delta)\mapsto\mathbf{SL}_2(R_0), R_0=\ker\delta$ fails to be forgetfully representable. Theorem~\ref{thm:exact} suggests that we should consider the functor\footnote{For a $k$--differential ring $\mathcal{R}=(R,\delta)$, $R_0=\ker\delta$ is called the ring of constants of $\mathcal{R}$ \cite{Gillet}.}:
$$\mathrm{cst}:k\text{\bf-drng}\rightarrow k\text{\bf-rng}, \quad \mathcal{R}=(R,\delta)\mapsto R_0=\ker(\delta).$$
It also yields two group functors:
\begin{align*}
&\mathbf{SL}_2^{\mathrm{cst}}:k\text{\bf-drng}\xrightarrow{\mathrm{cst}}k\text{\bf-rng}
\xrightarrow{\mathbf{SL}_2}\text{\bf grp}, &&\mathcal{R}=(R,\delta)\mapsto\mathbf{SL}_2(R_0).\\ &\boldsymbol{\mu}_2^{\mathrm{cst}}:k\text{\bf-drng}\xrightarrow{\mathrm{cst}}k\text{\bf-rng}
\xrightarrow{\boldsymbol{\mu}_2}\text{\bf grp},&&\mathcal{R}=(R,\delta)\mapsto\boldsymbol{\mu}_2(R_0).
\end{align*}
Furthermore, $\mathbf{SL}_2^{\mathrm{cst}}$ and $\boldsymbol{\mu}_2^{\mathrm{cst}}$ are representable as functors from $k\text{\bf-drng}$ to {\bf grp} in the following sense: $k[\mathbf{SL}_2]$ together with the zero derivation form a $k$--differential ring $(k[\mathbf{SL}_2],0)$. Similarly, $(k[\boldsymbol{\mu}_2],0)$ is a $k$--differential ring. Then it is easy to observe that
$$\mathbf{SL}_2^{\mathrm{cst}}=\mathrm{Hom}_{k\text{\bf-drng}}\left((k[\mathbf{SL}_2],0),-\right)
\text{ and }
\boldsymbol{\mu}_2^{\mathrm{cst}}=\mathrm{Hom}_{k\text{\bf-drng}}\left((k[\boldsymbol{\mu}_2],0),-\right).$$

\begin{proposition}
$\boldsymbol{\mu}_2^{\mathrm{fgt}}=\boldsymbol{\mu}_2^{\mathrm{cst}}$.
\end{proposition}
\begin{proof}
Let $\mathcal{R}=(R,\delta)$ be a $k$--differential ring and $r\in\boldsymbol{\mu}_2^{\mathrm{fgt}}(\mathcal{R})$, then $r^2=1$. Thus $2r\delta(r)=0$, and so $r\delta(r)=0$, from which we obtain $\delta(r)=rr\delta(r)=0$, i.e., $r\in R_0$. This yields that $r\in\boldsymbol{\mu}_2(R_0)=\boldsymbol{\mu}_2^{\mathrm{cst}}(\mathcal{R})$. Therefore, $\boldsymbol{\mu}_2^{\mathrm{fgt}}(\mathcal{R})\subseteq\boldsymbol{\mu}_2^{\mathrm{cst}}(\mathcal{R})$. The reverse inclusion is obvious.
\end{proof}

\section{The group functor $\mathbf{Aut}(\mathcal{F})$}
\label{sec:auto}

In this section, we consider the relationship between $\mathbf{Aut}(\mathcal{F})$ and $\mathbf{GrAut}(\mathcal{F})$.

\begin{theorem}
\label{thm:auto_int}
Let $\mathcal{R}=(R,\delta)$ be a $k$--differential ring with $R$ an integral domain. Then
$$\mathbf{GrAut}(\mathcal{F})(\mathcal{R})=\mathbf{Aut}(\mathcal{F})(\mathcal{R}).$$
\end{theorem}
\begin{proof}
Let $\varphi\in\mathbf{Aut}(\mathcal{F})(\mathcal{R})$. It suffices to show $\varphi(V\otimes_{k}R)\subseteq V\otimes_{k}R$.

Recall that $\mathcal{B}=k[\partial]\TT^1\oplus k[\partial]\TT^2\oplus k[\partial]\TT^3$. If we write
$$\varphi(\TT(\sigma^i))=\sum\limits_{m=0}^{M_i}\widehat{\partial}^mL(r_{im})+f_i,$$
$r_{im}\in R, f_i\in\mathcal{B}\otimes_k R$, $i=1,2,3$, then
\begin{align*}
&\varphi([\TT(\sigma^i)_\lambda\TT(\sigma^i)])=0,\\
&[\varphi(\TT(\sigma^i))_\lambda\varphi(\TT(\sigma^i))]\\
&\quad=\sum\limits_{m,n=0}^{M_i}(-\lambda)^m(\widehat{\partial}+\lambda)^n((\partial+2\lambda)\LL(r_{im}r_{in})+\LL(\delta(r_{im})r_{in}))+h_i,
\end{align*}
where $h_i\in k[\lambda]\otimes_k\mathcal{B}\otimes_k R$. By comparing the degree and coefficients of $\lambda$ in $$\varphi([\TT(\sigma^i)_\lambda\TT(\sigma^i)])=[\varphi(\TT(\sigma^i))_\lambda\varphi(\TT(\sigma^i))], i=1,2,3,$$ we obtain $M_i=0$ and $r_{iM_i}^2=0$, $i=1,2,3$. Thus $r_{iM_i}=0$ $i=1,2,3$ since $R$ is an integral domain, i.e., $\varphi(\TT(\sigma^i))\in\mathcal{B}\otimes_k R$. Since $\mathcal{B}\cong\mathrm{Cur}(\frak{sl}_2(k))$, by Corollary~3.17 in \cite{KLP}, $\varphi(\TT(\sigma^i))\subseteq (k\TT^1\oplus k\TT^2\oplus k\TT^3)\otimes_k R\subseteq V\otimes_k R$. More precisely, there is an $R$--linear automorphism of the Lie algebra $\alpha:\frak{sl}_2(R)\rightarrow\frak{sl}_2(R), \sigma^i\mapsto X^i, i=1,2,3$ such that $\varphi(\TT(\sigma^i))=\TT(X^i),i=1,2,3$.

A similar argument using $\varphi([\LL_\lambda\LL])=[\varphi(\LL)_\lambda\varphi(\LL)]$ yields that
$$\varphi(\LL)=\LL+\sum\limits_{m=0}^M\widehat{\partial}^m\TT(Y_m),$$
for $Y_m\in\frak{sl}_2(R)$. Then $\varphi([\LL_\lambda\TT(\sigma^i)])=[\varphi(\LL)_\lambda\varphi(\TT(\sigma^i))]$ implies $$[Y_m,X^i]=0,\text{ for } i=1,2,3, m>0.$$
Applying the automorphism $\alpha^{-1}$, we obtain
$$[\alpha^{-1}(Y_m),\sigma^i]=0,\text{ for } i=1,2,3,m>0.$$
Since $\{\sigma^i,i=1,2,3\}$ is a basis of $\frak{sl}_2(k)$ and $\alpha^{-1}(Y_m)\in\frak{sl}_2(R)$, we deduce that $\alpha^{-1}(Y_m)=0$ for $m>0$, and so $Y^m=0, m>0$. Hence, $\varphi(\LL)=\LL+\TT(Y_0)$, i.e., $\varphi(\LL)\in V\otimes_{k}R$.

Next we consider the odd part. By considering
$$\varphi([\LL_\lambda\GG(M)])=[\varphi(\LL)_\lambda\varphi(\GG(M))],$$ for $M\in \mathrm{Mat}_2(k)$, we obtain $\varphi(\GG(M))\in V\otimes_{k} R$. Hence, $\varphi\in\mathbf{GrAut}(\mathcal{F})(\mathcal{R})$.
\end{proof}

\begin{remark}
\label{rmk:nonintegral}
The assumption of integral domain in Theorem~\ref{thm:auto_int} is not superfluous. As an example, we consider the $k$--differential ring $\mathcal{R}=(k[\tau]/(\tau^2),0)$. In this situation, the $\mathcal{R}$--conformal superalgebra $\mathcal{F}\otimes_k\mathcal{R}$ has the automorphism $\varphi$ defined by
$$\varphi(\partial^{\ell}a\otimes r)=\partial^{\ell}a\otimes r+\partial^{\ell+1}a\otimes \bar{\tau}r,$$
for $\ell\in\Bbb{Z}_+, a\in V, r\in k[\tau]/(\tau^2),$ where $V$ is the $k$--vector space given in (\ref{eq:vectorspace}) and $\bar{\tau}$ is the canonical image of $\tau$ in $k[\tau]/(\tau^2)$. However, $\varphi$ is not contained in $\mathbf{GrAut}(\mathcal{F})(\mathcal{R})$, which shows that $\mathbf{Aut}(\mathcal{F})(\mathcal{R})\neq\mathbf{GrAut}(\mathcal{F})(\mathcal{R})$. An analogous example for the $N=2$ conformal superalgebra has been given in \cite[Remark~3.1]{CP}.
\end{remark}

\begin{corollary}
\label{cor:auto}
Let $\mathcal{R}=(R,\delta)$ be a $k$--differential ring such that $R$ is an integral domain and $\mathrm{H}_{\text{\'{e}t}}^1(R,\boldsymbol{\mu_2})$ is trivial. Then
$$\mathbf{Aut}(\mathcal{F})(\mathcal{R})
\cong\frac{\mathbf{SL}_2(R)\times\mathbf{SL}_2(R_0)}{\boldsymbol{\mu}_2(R_0)},$$
where $R_0=\ker\delta$.
\end{corollary}

\begin{remark} \label{rmk:examples}\quad
\begin{itemize}
\item Since $k$ is an algebraically closed field of characteristic zero, $\mathrm{H}_{\text{\'{e}t}}^1(k,\boldsymbol{\mu_2})$ is trivial. Hence,
$$\mathrm{Aut}_{k\text{-conf}}(\mathcal{F})=\mathbf{Aut}(\mathcal{F})(k,0)\cong\frac{\mathbf{SL}_2(k)\times\mathbf{SL}_2(k)}{\langle(-I_2,-I_2)\rangle}.$$
\item For the $k$-differential ring $\widehat{\mathcal{D}}=(k[t^q|q\in\Bbb Q],\frac{d}{dt})$, the \'{e}tale cohomology set $\mathrm{H}_{\text{\'{e}t}}^1(\widehat{D},\boldsymbol{\mu_2})$ is trivial (\cite[Theorem~2.9]{GiP}) and $\widehat{D}_0=\ker(\frac{d}{dt})=k$. Hence,
    $$\mathrm{Aut}_{\widehat{\mathcal{D}}\text{-conf}}(\mathcal{F}\otimes_k\widehat{\mathcal{D}})=
    \mathbf{Aut}(\mathcal{F})(\widehat{\mathcal{D}})=\frac{\mathbf{SL}_2(\widehat{D})\times\mathbf{SL}_2(k)}{\langle(-I_2,-I_2)\rangle},$$
    which is the result of Proposition~3.69 in \cite{KLP}.
\end{itemize}
\end{remark}

\section{Comments on the functor $\mathbf{GrAut}(\mathcal{A})$}
\label{sec:grautodef}
Let $\mathcal{A}$ be an arbitrary Lie conformal superalgebra over $k$ such that $\mathcal{A}$ is a free $k[\partial]$--module of finite rank, then there is a finite-dimensional $k$--vector space $W$ such that $\mathcal{A}=k[\partial]\otimes_{k}W$ as $k[\partial]$--modules. Hence, we can define a subgroup of $\mathbf{Aut}(\mathcal{A})(\mathcal{R})$ for each $k$--differential ring $\mathcal{R}=(R,\delta)$:
$$\mathbf{GrAut}(\mathcal{A},W)(\mathcal{R})=\{\phi\in\mathbf{Aut}(\mathcal{A})(\mathcal{R})|\phi(k\otimes_k W\otimes_k R)\subseteq k\otimes_k W\otimes_k R\}.$$
This construction is also functorial in $\mathcal{R}$. Hence, $\mathbf{GrAut}(\mathcal{A},W)$ is also a functor from the category of $k$--differential rings into the category of groups. However, $W$ is not uniquely determined by $\mathcal{A}$. In other words, the group functor $\mathbf{GrAut}(\mathcal{A},W)$ may depend on the choice of $W$.

According to \cite[section 5.10]{K}, the $N=1,2,3$ Lie conformal superalgebra $\mathcal{K}_N$ can be realized as follows: $\mathcal{K}_N=k[\partial]\otimes_{k}\Lambda(N)$, where $\Lambda(N)$ is the Grassman superalgebra in $N$ variables, $\partial$ acts on $\mathcal{K}_N$ in the natural way, and the $n$-th product for each $n\in\Z_+$ is defined by
\begin{align*}
f_{(0)}g&=\left(\frac{1}{2}|f|-1\right)\partial\otimes fg+\frac{1}{2}(-1)^{|f|}\sum\limits_{i=1}^N(\partial_if)(\partial_i g)\\
f_{(1)}g&=\left(\frac{1}{2}(|f|+|g|)-2\right)fg,\\
f_{(n)}g&=0, \quad n\geqslant2,
\end{align*}
where $f,g\in\Lambda(N)$ are homogenous polynomials in $\xi_1,\cdots,\xi_N$ of degree $|f|$ and $|g|$ respectively, and $\partial_i$ is the derivative with respect to $\xi_i, i=1,\cdots,N$.

For the $N=1,2,3$ Lie conformal superalgebras $\mathcal{K}_N$, the particular choice that $W=\Lambda(N)$ has been considered in \cite{CP}. In this situation, $$\mathbf{GrAut}(\mathcal{K}_N,\Lambda(N))(\mathcal{R})=\mathbf{Aut}(\mathcal{K}_N)(\mathcal{R}),\quad N=1,2,3,$$
for every $k$--differential ring $\mathcal{R}=(R,\delta)$ with $R$ an integral domain.

For the $N=4$ Lie conformal superalgebra $\mathcal{F}$, we take $W=V$ to be the $k$--vector space as in (\ref{eq:vectorspace}) in Section~\ref{sec:basic}. It is known from Theorem~\ref{thm:auto_int} that $$\mathbf{GrAut}(\mathcal{F},V)(\mathcal{R})=\mathbf{Aut}(\mathcal{F})(\mathcal{R})$$
for every $k$--differential ring $\mathcal{R}=(R,\delta)$ with $R$ an integral domain.

However, for a given Lie conformal superalgebra $\mathcal{A}$ over $k$ and an arbitrary choice of $W$ such that $\mathcal{A}=k[\partial]\otimes_kW$, $\mathbf{GrAut}(\mathcal{A},W)(\mathcal{R})\subseteq\mathbf{Aut}(\mathcal{A})(\mathcal{R})$ may fail to be an equality even if $R$ is the field $k$.

For instance, considering the $N=2$ Lie conformal superalgebra $\mathcal{K}_2$, we have known from \cite{CP} that $W=\Lambda(2)$ is a good choice to define $\mathbf{GrAut}(\mathcal{K}_2,\Lambda(2))$. However, $\mathcal{K}_2$ is also realized as follows (see \cite[section 5.10]{K}):
$$\mathcal{K}_2=k[\partial]\otimes_{k}(\mathrm{Der}(\Lambda(1))\oplus\Lambda(1)),$$
where $\Lambda(1)=k\oplus k\xi$ is the Grassman superalgebra in one variable and $\mathrm{Der}(\Lambda(1))$ is the superalgebra of all derivations of $\Lambda(1)$. The $n$-th products for $n\in\Z_+$ on $\mathcal{K}_2$ are given by
$$\begin{array}{l}
a_{(n)}b=\delta_{n0}[a,b],\quad a_{(0)}f=a(f),\quad a_{(n)}f=-\delta_{n1}(-1)^{p(a)p(f)}fa,\text{ if }n\geqslant1,\\
f_{(0)}g=-\partial(fg),\quad f_{(n)}g=-2\delta_{j1}fg,\text{ if }n\geqslant1,
\end{array}$$
where $a,b\in\mathrm{Der}(\Lambda(1)), f,g\in\Lambda(1)$, and $p(a)$ (resp. $p(b)$) is the parity of $a$ (resp. $b$).

If we choose
\begin{eqnarray*}
W=W':=\mathrm{Der}(\Lambda(1))\oplus\Lambda(1)=\left(k\frac{d}{d\xi}\oplus k\xi\frac{d}{d\xi}\right)\oplus(k\oplus k\xi),
\end{eqnarray*}
then $\mathcal{K}_2=k[\partial]\otimes_{k}W'$. For such a $k$--vector subspace $W'\subseteq\mathcal{K}_2$, the subgroup functor $\mathbf{GrAut}(\mathcal{K}_2,W')$ is not necessarily equal to $\mathbf{Aut}(\mathcal{K}_2)$ when evaluating at a $k$-differential ring $\mathcal{R}=(R,\delta)$ with $R$ an integral domain. Indeed, the Lie conformal superalgebra $\mathcal{K}_2$ has the automorphism $\phi:\mathcal{K}_2\rightarrow\mathcal{K}_2$ defined by
$$\phi(1)=1-\partial\xi\frac{d}{d\xi}, \quad\phi(\xi)=\frac{d}{d\xi}, \quad\phi\left(\frac{d}{d\xi}\right)=\xi, \quad\phi\left(\xi\frac{d}{d\xi}\right)=-\xi\frac{d}{d\xi}.$$
Although $(k,0)$ is a $k$--differential ring and $k$ is a field, the automorphism $\phi$ is not contained $\mathbf{GrAut}(\mathcal{K}_2,W')(k)$.
Hence, $W'$ is not a suitable choice to define $\mathbf{GrAut}(\mathcal{K}_2)$.

\section*{Acknowledgments}
The author is grateful to Professor Arturo Pianzola for his valuable suggestions and to Valerie Budd for her careful reading of the manuscript. He thanks the referees for their helpful comments. Finally, the author appreciates the support of the University of Alberta and the Chinese Scholarship Council.

\end{document}